\documentclass[12pt]{amsart}
\usepackage{amscd,amssymb,amsthm,amsmath,amssymb,mathrsfs,enumerate}
\usepackage[matrix,arrow]{xy}
\usepackage[margin=2cm]{geometry}
\usepackage{tikz-cd}

\usepackage{biblatex} 
\newtheorem{theorem}{Theorem}

\newtheorem{lemma}[theorem]{Lemma}
\newtheorem{corollary}[theorem]{Corollary}
\newtheorem{conjecture}[theorem]{Conjecture}

\newtheorem*{maintheorem*}{Main Theorem}
\newtheorem*{classificationtheorem*}{Classification Theorem}
\newtheorem*{technicallemma1*}{Technical Lemma 1}
\newtheorem*{technicallemma2*}{Technical Lemma 2}

\theoremstyle{definition}

\theoremstyle{remark}
\newtheorem{remark}[theorem]{Remark}

\newcommand{\Gm}{\mathbb{G}_m}
\newcommand{\Ga}{\mathbb{G}_a}

\newcommand{\F}{\mathbb{F}}
\renewcommand{\P}{\mathbb{P}}
\newcommand{\Z}{\mathbb{Z}}
\newcommand{\C}{\mathbb{C}}
\newcommand{\ol}[1]{\overline{#1}}

\renewcommand{\O}{\mathcal{O}}
\newcommand{\Q}{\mathbb{Q}}
\newcommand{\R}{\mathbb{R}}

\title{K-stable Fano threefolds of rank $2$ and degree $28$}

\author{Joseph Malbon}

\thanks{Throughout this paper, all varieties are assumed to be projective and defined over $\mathbb{C}$.}

\address{\emph{Joseph Malbon}
\newline
\textnormal{University of Edinburgh, Edinburgh, United Kingdom}
\newline
\textnormal{\texttt{j.malbon@sms.ed.ac.uk}}}

\begin{document}

\begin{abstract}
    Moduli spaces of Fano varieties have historically been difficult to construct. However, recent work has shown that smooth K-polystable Fano varieties of fixed dimension and volume can be parametrised by a quasi-projective moduli space. In this paper, we prove that all smooth Fano threefolds with Picard rank 2 and degree 28 are K-polystable, except for some explicit cases which we describe.
\end{abstract}

\maketitle

\section{Introduction}
\label{section:intro}


Fano varieties, those which have ample anti-canonical class, are a fundamental class of varieties in algebraic geometry. The minimal model program predicts that any smooth variety is birational to either a canonically polarised variety, or to a variety fibred by either Fano varieties or varieties with numerically trivial canonical class. Thus, they form one of three classes of ``building blocks" from which all varieties may be constructed. Many varieties which are defined by polynomials of low degree are Fano, such as hypersurfaces in $\P^n$ of degree $\leq n$, or more generally complete intersections in $\P^n$ of hypersurfaces $H_1,\ldots,H_r$ such that $\sum_i\mathrm{deg}(H_i)\leq n$.  \\

It was proven by Koll\'ar, Miyaoka and Mori in \cite{KMM} that smooth Fano varieties of given dimension arise in finitely many deformation families. Namely, for each natural number $n$, there are finitely many smooth, projective morphisms $\mathcal{X}_i\to T_i$ of quasi-projective varieties such that whenever $X$ is a smooth $n$-dimensional Fano variety, then $X$ is isomorphic to a fibre of $\mathcal{X}_i\to T_i$, for some $i$.\\ 

For example, the projective line $\P^1$ is the only Fano curve. The Fano, or \textit{del Pezzo}, surfaces are $\P^1\times\P^1$, and $\P^2$ blown up in at most 8 points in sufficiently general position, and the members of each of these 10 families are parametrised by an irreducible variety.  Fano threefolds were classified by Iskovskikh \cite{Iskovskikh}, Mori and Mukai \cite{MoriMukai}, into $105$ deformation families labeled as \textnumero 1.1, \textnumero 1.2, \textnumero 1.3,   $\ldots$, \textnumero 10.1, and threefolds in each of these $105$ families are parametrised by an irreducible variety.  \\ 

Historically the construction of moduli spaces for each deformation family seemed out of reach, due in part to the fact that Fano varieties often have complicated degenerations, and so separateness of any sort of space is not achievable. However, it  was shown in \cite{BX} that degenerations of Fano varieties which are K-polystable are indeed unique. Furthermore, it has been shown recently in \cite{Jiang, BLX, BX}, and in \cite{ABHLX}, that the moduli stack of families of K-semistable $\Q$-Fano varieties of given dimension and volume admits a good moduli space (in the sense of Alper, \cite{Alper}) which is a separated algebraic space whose closed points are in bijection with the \textit{K-polystable} varieties. Moreover, it was shown in \cite{LXZ} that this algebraic space is in fact a projective variety. Thus, a natural question to ask is which elements in each deformation family, if any, are K-semistable and K-polystable. We refer the reader to the survey \cite{Xu} and the book \cite{Book}. \\ 

This question has been answered in many cases, for example in \cite{AbbanZhuangSeshadri,Book,BelousovLoginov,CheltsovDenisovaFujita,CheltsovFujitaKishimotoOkada,CFKP,CheltsovPark,Denisova,GGV,Liu,XuLiu}. The main goal of this paper is to prove a conjecture stated in \cite[\S~7.3]{Book},
which describes all K-polystable smooth Fano threefolds in the~deformation family \textnumero 2.21. \\

Let us now describe this family. Let $Q\subset\P^4$ be a smooth quadric threefold and let $C_4$ be a rational normal quartic curve $C_4\subset Q$. Let $\pi\colon X\to Q$ be the blow-up of $Q$ in $C_4$. Family \textnumero2.21 consists of all threefolds obtained this way. The threefold $X$ has $\mathrm{Pic}(X)\cong\Z^2$ and volume $(-K_X)^3=28$, and moreover every smooth Fano threefold with these invariants is a member of family \textnumero2.21. \\ 

One easily sees that the map $H^0(\P^4,\O_{\P^4}(2))\to H^0(C_4,\O_{C_4}(2))$ is surjective, so from the ideal sheaf exact sequence for $C_4\subset\P^4$ and the isomorphism $\O_{C_4}(2)\cong\O_{\P^1}(8)$, we have that the vector space $H^0(\P^4, \mathcal{I}_{C_4/\P^4}(2))$ has dimension $6=15-9$. Thus, quadric threefolds that contain $C_4$ are parametrised by $\P^5$. The natural $\mathrm{SL}_2(\C)$-action on $C_4$ lifts to $\P^4$ in such a way that $C_4$ is invariant. Then $\mathrm{SL}_2(\C)$ acts naturally on this parameter space. Now, we are ready to present the conjecture stated in~\mbox{\cite[\S~7.3]{Book}}.
\begin{conjecture}[{\cite[\S~7.3]{Book}}]
\label{conjecture:main}
The smooth Fano threefold $X$ is K-polystable if and only if the quadric $Q$ is GIT-polystable with respect to the $\mathrm{SL}_2(\C)$-action.
\end{conjecture}
\begin{remark}
\label{remark:GIT}
Let us now describe what is known about this problem. Firstly, the GIT stability for the action of $\mathrm{SL}_2(\C)$ on the parameter space is well-understood (for example, see \cite[Corollary A.58]{Book}). To describe this, let us fix a coordinate system: up to projective transformation, $C_4$ is the image of the degree 4 Veronese embedding of $\P^1$:
$$[u:v] \mapsto[u^4:u^3v:u^2v^2:uv^3:v^4].$$
Then the vector space $H^0(\P^4,\mathcal{I}_{C_4/\P^4}(2))$ is generated by the following quadratic forms:
\begin{align*}
f_0&=x_3^2-x_2x_4,\\
f_1&=x_2x_3-x_1x_4,\\
f_2&=x_2^2-x_0x_4,\\
f_3&=x_1x_2-x_0x_3,\\
f_4&=x_1^2-x_0x_2,\\
f_5&=3x_2^2-4x_1x_3+x_0x_4, 
\end{align*}
where $x_0$, $x_1$, $x_2$, $x_3$, $x_4$ are homogeneous coordinates on $\mathbb{P}^4$. \\ 

We have that $\{f_5=0\}$ is the unique $\text{SL}_2(\mathbb{C})$-invariant quadric threefold in $\mathbb{P}^4$. Thus, the $\text{SL}_2(\C)$-representation $\P^5\cong\P\big(H^0(\mathcal{I}_{C_4/\P^4}(2))\big)$ splits as a direct sum: $\P(V\oplus\C)$, where $V$ is the five-dimensional vector space spanned by $f_0,f_1,f_2,f_3,f_4$, and $\text{SL}_2(\mathbb{C})$ acts trivially on the summand $\C$ spanned by $f_5$. Then the GIT-stability for this representation is as follows: every point in $\mathbb{P}^5$ is stable, except for
the following five explicit cases, up to the action of $\text{SL}_2(\C)$:
\begin{center}
\renewcommand\arraystretch{1.7}
\begin{tabular}{|c|c|c|}
\hline
\quad Quadric threefold $Q\subset\mathbb{P}^4$ \quad  & \quad GIT-semistability \quad & \quad  $\mathrm{Aut}(Q,C_4)$\quad  \\
\hline
\hline
$\{f_5=0\}$ & GIT-polystable & $\mathrm{PGL}_2(\mathbb{C})$ \\
\hline
$\{\lambda f_2+f_5=0\}$ for $\lambda\in\mathbb{C}\setminus\{0,1,-3\}$& GIT-polystable & $\Gm\rtimes\Z_2$ \\
\hline
$\{f_0+f_5=0\}$ & strictly GIT-semistable & $\Ga\rtimes\Z_2$ \\
\hline
$\{f_0+\lambda f_2+ f_5=0\}$ for $\lambda\in\mathbb{C}\setminus\{0,1,-3\}$ &  strictly GIT-semistable & finite \\
\hline
$\{f_1+f_5=0\}$ & strictly GIT-semistable & finite \\
\hline
\end{tabular}
\end{center}

Here $\mathrm{Aut}(Q,C_4)$ denotes the group of automorphisms of $Q$ which leave $C_4$ invariant, and $\Z_2$ denotes the cyclic group of order 2. It is known (see \cite[Lemma 9.2]{CPS},\cite[\S~5.9]{Book}) that the group $\mathrm{Aut}(Q,C_4)$ is finite except for three cases which are described in the table above. The action of $\Gm$ on $Q$ in the first two cases of the table is given by the following action on $\P^4$:
$$t\colon [x_0:x_1:x_2:x_3:x_4]\mapsto[x_0:tx_1:t^2x_2:t^3x_3:t^4x_4].$$
In these two cases, the threefold $X$ is known to be K-polystable (see \cite[\S~5.9]{Book}). For the other three cases, the threefold $X\to Q$ is strictly K-semistable (K-semistable but not K-polystable). To see this, note that the closure of $\Gm$-orbit of $Q$ under the above action (considered to be acting on the parameter space $\P^5$) contains a point corresponding to the quadric:
$$Q_\lambda:=\big\{\lambda f_2+f_5=0\big\},$$
for some $\lambda\in\C\setminus\{1,-3\}$. This defines an isotrivial degeneration of $X$ to a K-polystable threefold $X_\lambda$, which is the blow-up of the quadric $Q_\lambda$. Thus $X$ is K-semistable by \cite{BLX}, but is not K-polystable by \cite{BX} since $\mathrm{Aut}(X)\ncong\mathrm{Aut}(X_\lambda)$\footnote{One can also see that the threefold corresponding to the equation $f_0+f_5=0$ is not K-polystable because K-polystable Fano threefolds have reductive automorphism groups \cite{Mat57,ABHLX}.} implies that $X\ncong X_\lambda$. See \cite[Corollary 1.13]{Book} for a more detailed explanation of this phenomenon.
\end{remark}

Now let us turn to the proof of Conjecture \ref{conjecture:main}. In this paper we prove the following two theorems, which imply this conjecture:

\begin{classificationtheorem*}
\label{classification}
There is $\phi\in\mathrm{SL}_2(\mathbb{C})$ such that one of the following cases holds:
\begin{enumerate}
\item[$(\mathrm{1})$] $\phi^*(Q)=\{\mu(f_0+f_4)+\lambda f_2+f_5=0\}$ for some $\lambda,\mu\in\mathbb{C}$ such that
$$(\mu - 2)(\mu + 2)(\lambda - 1)(\lambda^2 + \mu^2 + 2\lambda - 3)\neq0,$$
\item[$(\mathrm{2})$] $\phi^*(Q)=\{f_0+\lambda f_2+f_5=0\}$ for some $\lambda\in\mathbb{C}\setminus\{1,-3\}$,
\item[$(\mathrm{3})$] $\phi^*(Q)=\{f_1+f_5=0\}$.
\end{enumerate}
\end{classificationtheorem*}

\begin{maintheorem*}
If the group $\mathrm{Aut}(Q,C_4)$ contains a subgroup isomorphic to $\Z_2\times\Z_2$, then $X$ is K-polystable.
\end{maintheorem*}

\begin{proof}[Proof of Conjecture~\ref{conjecture:main}]
Suppose that $Q$ is GIT-polystable. Then by the Classification Theorem and the classification of GIT-stability discussed in Remark \ref{remark:GIT}, we may assume that it is given by an equation of the form $(1)$. Then the group $\mathrm{Aut}(Q,C_4)$ has a subgroup isomorphic to $\Z_2\times\Z_2$, which is generated by the involutions
\begin{align*}
    \iota\colon [x_0:x_1:x_2:x_3:x_4]&\mapsto[x_4:x_3:x_2:x_1:x_0], \\ 
    \tau\colon[x_0:x_1:x_2:x_3:x_4]&\mapsto [x_0:-x_1:x_2:-x_3:x_4],
\end{align*}
so that $X$ is K-polystable by the Main Theorem.

Now suppose that $Q$ is strictly GIT-semistable. Then we may assume that $Q$ is given by one of the equations $(2)$ and $(3)$, so that $X$ is strictly K-semistable by the discussion at the end of Remark \ref{remark:GIT}.
\end{proof}

\begin{remark}
\label{remark:Sarkisov}
Our proof will make use of the following Sarkisov link, which will be explained in Lemma \ref{sarkisov}:
$$
\xymatrix{
&X\ar[ld]_{\pi}\ar[rd]^{\pi^\prime}&\\
Q\ar@{-->}[rr]_{\phi}&&Q^{\prime}.}
$$
Breifly, $Q'$ is a smooth quadric threefold in $\P^4$, $\phi$ is the rational map given by the linear system of quadric sections of $Q$ which contain $C_4$, and $\pi'$ is the blow-up of $Q'$ in a rational normal quartic curve $C_4'\subset Q'$, with exceptional divisor $E'$. Then $\phi$ contracts secants of $C_4$ which are contained in $Q$.  Let
$$V_3=\{ x_0x_3^2 - 2x_1x_2x_3 - x_0x_2x_4 + x_1^2x_4 + x_2^3=0\}\subset\P^4.$$
Then $V_3$ is spanned by the secants of $C_4$, and $\phi$ contracts the surface $V_3\cap Q$. Moreover, $V_3\cap Q$ is the unique non-normal $(2,3)$-complete intersection which has multiplicity two along $C_4$, whose equation can be found by hand. We have that the $\pi'$-exceptional divisor $E'$ is the strict transform of $V_3\cap Q$ with respect to $\pi$. Thus follows the linear equivalence $E'\sim 3H-2E$.
\end{remark}

Let us describe the structure of this paper. In Section~\ref{section:Abban-Zhuang}, we present several results from \cite{AbbanZhuang,Book} which are used in the proof of the Main Theorem. In Section~\ref{section:basic-properties}, we present some basic properties of threefolds in the deformation family \textnumero 2.21 which will be crucial in the proof: namely some properties of the exceptional divisor $E$ of $\pi\colon X\to Q$, and of pullbacks by $\pi$ of hyperplane sections of $Q$. In Section~\ref{section:classification} we prove the Classification Theorem for elements of \textnumero 2.21, and then in Section~\ref{section:technical} we prove the following two key results:

\begin{technicallemma1*}
Let $P$ be a point in $X$ which is not contained in $E\cup E'$.
Then $\delta_P(X)>1$.
\end{technicallemma1*}

\begin{technicallemma2*}
Let $D$ be a prime divisor over $X$ whose centre on $X$ is a curve $C$ contained in $E\cup E'$.
Suppose, in addition, that neither $\pi(C)$ nor $\pi^\prime(C)$ is a point. 
Then $A_X(D)>S_X(D)$. 
\end{technicallemma2*}

Finally, using these results, we prove the Main Theorem in Section~\ref{section:proof}. In Section \ref{section:normal bundles}, we present a classification result on the normal bundle of the rational normal quartic curve in a quadric, which refines Lemma \ref{lemma:normal-bundle} in Section \ref{section:basic-properties}.\\ 

\medskip
\noindent
\textbf{Acknowledgements.}
I am grateful to I. Cheltsov for introducing me to this topic, and for his attentive guidance, support and insight with regards to both the creation of this document, and my study of the subject in general.

\section{Abban--Zhuang theory \& Fujita's formulas}
\label{section:Abban-Zhuang}

In this section we present several key formulas from Abban--Zhuang theory discovered in \cite{AbbanZhuang,Book}. \\

Let $X$ be a smooth Fano threefold, $D$ be a prime divisor over $X$ (via a birational morphism $f\colon\widetilde{X}\to X$ with $D\subset\widetilde{X}$). Write $A_X(D)=1+\text{ord}_D(K_{\widetilde{X}}-f^*(K_X))$ for the log discrepancy of $D$ with respect $X$, and $$S_X(D)=\frac{1}{(-K_X)^3}\int_{0}^{\infty}\mathrm{vol}\big(f^*(-K_X)-uD\big)du.$$
Let $S$ be a smooth surface in $X$.
Set

$$
\tau=\mathrm{sup}\Big\{u\in\mathbb{R}_{\geqslant 0}\ \big\vert\ \text{the divisor  $-K_X-uS$ is pseudo-effective}\Big\}.
$$
For $u\in[0,\tau]$, let $P(u)$ be the~positive part of the~Zariski decomposition of the~divisor $-K_X-uS$,
and let $N(u)$ be its negative part. For any prime divisor $Y$ over $S$ (via a birational morphism $g\colon\widetilde{S}\to S$, with $Y\subset\widetilde{S}$), we let $\widetilde{N}(u)$ be the strict transform on $\widetilde{S}$ of the divisor $N(u)\vert_{S}$. For $u\in[0,\tau]$, we let
$$
\widetilde{t}(u)=\sup\Big\{v\in \mathbb{R}_{\geqslant 0} \ \big| \ g^*\big(P(u)|_S\big)-vY \text{ is pseudo-effective}\Big\}.
$$
Let $\widetilde{P}(u,v)$ and $\widetilde{N}(u,v)$ be the positive and negative parts, respectively, of the Zariski decomposition of the divisor $g^*\big(P(u)|_S\big)-vY$. We set
$$
S\big(W^S_{\bullet,\bullet};Y\big)=\frac{3}{(-K_X)^3}\int_0^\tau\big(P(u)\big\vert_{S}\big)^2\cdot\mathrm{ord}_{Y}\big(g^*(N(u)\big\vert_{S})\big)du+\frac{3}{(-K_X)^3}\int_0^\tau\int_0^{\widetilde{t}(u)} \widetilde{P}(u,v)^2dvdu.
$$
Firstly we have the following estimate, which will be used in the proof of Theorem \ref{lemma:technical2}:
\begin{theorem}[{\cite{AbbanZhuang} and\cite[Theorem 1.112]{Book}}]
\label{theorem:Hamid-Ziquan-Kento-1}
Suppose $X$ contains an irreducible normal surface $S$, which in turn contains an irreducible curve $C$. Then for any prime divisor $D$ over $X$ with centre $C$, we have
$$
\frac{A_X(D)}{S_X(D)}\geqslant\min\Bigg\{\frac{1}{S_X(S)},\frac{1}{S\big(W^S_{\bullet,\bullet};C\big)}\Bigg\}.
$$
\end{theorem}

Now let $P$ be a point of $S$, and let $g\colon\widetilde{S}\to S$ be the blow-up of the surface $S$ at $P$. Let $Y$ be the $g$-exceptional curve. For every point $O\in Y$, we set
$$
S\big(W_{\bullet, \bullet,\bullet}^{S,Y};O\big)=
\frac{3}{(-K_X)^3}\int_0^\tau\int_0^{\widetilde{t}(u)}\big(\widetilde{P}(u,v)\cdot Y\big)^2dvdu+
F_O\big(W_{\bullet,\bullet,\bullet}^{S,Y}\big),
$$
where
$$
F_O\big(W_{\bullet,\bullet,\bullet}^{S,Y}\big)=
\frac{6}{(-K_X)^3}\int_0^\tau\int_0^{\widetilde{t}(u)}\big(\widetilde{P}(u,v)\cdot Y\big)\times\mathrm{ord}_O\big(\widetilde{N}(u)\big|_Y+\widetilde{N}(u,v)\big|_{Y}\big)dvdu.
$$

The \textit{local delta invariant} of $X$ at a point $P\in X$ is defined as 
$$\delta_P(X)\,\,\,=\inf_{\substack{D/X\\P\in c_X(D)}}\frac{A_X(D)}{S_X(D)},$$
where the infimum runs over all prime divisors $D$ over $X$ with centre, denoted $c_X(D)$, containing $P$. Then we have the following estimate, which will be used in the proof of Theorem \ref{lemma:technical1}:
\begin{theorem}[{\cite{AbbanZhuang} and \cite[Corollary 1.110]{Book}}]
\label{theorem:Hamid-Ziquan-Kento-3}
One has
$$
\delta_P(X)\geqslant\min\Bigg\{\frac{1}{S_X(S)},\frac{2}{S\big(W^S_{\bullet,\bullet};Y\big)},\inf_{O\in Y}\frac{1}{S\big(W_{\bullet, \bullet,\bullet}^{S,Y};O\big)}\Bigg\}.
$$
\end{theorem}

\begin{remark}
\label{remark:6}
    If $P\not\in\mathrm{Supp}(N(u))$ for every $u\in[0,\tau]$, then $\text{ord}_Y(g^*(N(u)|_S))=0$, so that the formulas for $S(W_{\bullet, \bullet}^{S};Y)$ and
$F_O(W_{\bullet,\bullet,\bullet}^{S,Y})$ simplify as
\begin{align*}
S\big(W_{\bullet,\bullet}^{S};Y\big)&=\frac{3}{(-K_X)^3}\int_0^\tau\int_0^{\widetilde{t}(u)}\big(\widetilde{P}(u,v)\big)^2dvdu,\\
F_O\big(W_{\bullet,\bullet,\bullet}^{S,Y}\big)&=\frac{6}{(-K_X)^3}\int_0^\tau\int_0^{\widetilde{t}(u)}\big(\widetilde{P}(u,v)\cdot Y\big)\times\mathrm{ord}_O\big(\widetilde{N}(u,v)\big|_{Y}\big)dvdu.
\end{align*}
\end{remark}

\section{Properties of smooth Fano threefolds of rank $2$ and degree $28$}
\label{section:basic-properties}

In the following, $C_4$ is the smooth rational quartic curve in $\P^4$ given by the parametrisation
$$
[u:v] \mapsto[u^4:u^3v:u^2v^2:uv^3:v^4],
$$
and $Q\subset\P^4$ is any smooth quadric containing $C_4$. Let $\phi\colon Q\dashrightarrow\P^4$ be the rational map given by the linear system of quadric sections of $Q$ which contain $C_4$. We now prove the assertions in Remark \ref{remark:Sarkisov}.

\begin{lemma}
\label{sarkisov}
   Let $\pi\colon X\to Q$ be the blow-up of $Q$ in $C_4$. Then there is a commutative diagram:
   $$
\xymatrix{
&X\ar[ld]_{\pi}\ar[rd]^{\pi^\prime}&\\
Q\ar@{-->}[rr]_{\phi}&&Q^{\prime},}
$$
where $Q'$ is a smooth quadric threefold in $\P^4$ and $\pi'$ is the blow-up of $Q'$ in a rational normal quartic curve.
\end{lemma}

\begin{proof}
Let $H$ be the pullback by $\pi$ of a hyperplane section of $Q$, let $\pi'$ by the rational map given by the linear system $|2H-E|$, and let $Q'$ by the closure of the image of $\phi$. Since $C_4$ is a scheme-theoretic intersection of quadrics, we have that the linear system $|2H-E|$ is free, so that the map $\pi'$ is a morphism, and moreover that the above diagram commutes since the linear system $|2H-E|$ consists of the strict transforms of quadrics which contain $C_4$. Since the map $\phi$ is given by quadrics, it contracts the secant variety $V_3\cap Q$ of $C_4$, so that $\pi'$ contracts its strict transform, $E'$, onto some subvariety $C_4'\subset Q'$ of dimension at most 1.\\ 

We now describe the Chow ring of $X$. Let $\mathbf{f}$ denote the rational equivalence class in $X$ of a fibre of $\pi$ over a general point of $C_4$, and let $\mathbf{l}$ denote the class pullback by $\pi$ of a line in $Q$. Then the following properties are well-known (see for example \cite[Lemma 2.2.14]{Isk}):
\begin{align*}
 E\cdot\pi^*(D)&\sim(C_4\cdot D)\mathbf{f},\\
 E\cdot\pi^*(\mathbf{z})&=0, \\
\mathbf{f}\cdot\pi^*(D)&=0,
\end{align*}
for all divisors $D$ and algebraic 1-cycles $\mathbf{z}$ on $Q$. Furthermore, we have 
\begin{align*}
    E^3&=-\text{deg}(N_{C_4/Q}), \\ 
    E^2&\sim-\pi^*(C_4)+\text{deg}(N_{C_4/Q})\mathbf{f}, \\ 
    E\cdot\mathbf{f}&=-1,
\end{align*}
where the degree of the normal bundle of $C_4$ in $Q$ can be computed as:
\begin{align*}
    \text{deg}(N_{C_4/Q})&=2g(C_4)-2-K_Q\cdot C_4=10.
\end{align*}

In particular, we have the rational equivalence $E^2\sim10\mathbf{f}-4\mathbf{l}$, and the following intersection numbers on $X$:
$$H^3=2,\quad H^2\cdot E=0,\quad H\cdot E^2=-4,\quad E^3=-10.$$

Then $(2H-E)^3=2$, so that the morphism $\pi'$ has three-dimensional image and is either birational onto a quadric, or a double cover of a hyperplane in $\P^4$. But the image of $\phi$ is not contained in any hyperplane of $\P^4$, so that $\pi'$ must be birational onto an irreducible quadric $Q'$. 

Let $\mathbf{f}'$ denote the class of the strict transform of a secant of $C_4$ with respect to $\pi$, and observe that $\mathbf{f}'\sim\mathbf{l}-2\mathbf{f}$. Since $\pi'$ contracts $\mathbf{f}'$ to a point and $X$ is a Fano variety of Picard rank 2, we have that $\mathbf{f}'$ spans a $K_X$-negative extremal ray $R\subset\overline{\mathrm{NE}}(X)$ which consists precisely of curve classes contracted by $\pi'$. Moreover since $\pi'$ is birational and $Q'$ is normal we have $(\pi')_*\O_X\cong\O_{Q'}$, so that $\pi'$ is the extremal contraction of $R$. In particular, the exceptional locus of $\pi'$ is precisely $E'$.

We now show that $C_4'$ is a curve. For this, notice that $E|_{E'}$ is a curve contained in $E'$. Then we have 
\begin{align*}
    (2H-E).E|_{E'} = (2H-E).E.(3H-2E) =8>0,
\end{align*}
so that $\pi(E|_{E'})$, and hence $C_4'$, is a curve. Then by the classification of $K_X$-negative extremal contractions on a smooth threefold, $\pi'\colon X\to Q'$ is the blow-up along $C_4'$, and $Q'$ and $C_4'$ are smooth. \\
    
The last thing to show is that $C_4'$ is indeed a rational normal quartic curve. From the linear equivalence $E'\sim3H-2E$, we have that:
$$
    2g(C_4')-2-K_{Q'}.C_4'=\mathrm{deg}N_{C_4'/Q'} 
    =-(E')^3 
    = 10.
$$
Combining this with the adjunction formula $K_{Q'}=-3H'$ gives the following Diophantine equation:
$$2g(C_4')+3H'.C_4'=12.$$
Then since $g(C_4')$ is non-negative and $H'.C_4'$ is positive, we have the two solutions $g(C_4')=0,H'.C_4'=4$ and $g(C_4')=3,H'.C_4'=2$. The latter is obviously impossible, so we must have that $C_4'$ is a rational normal quartic.
\end{proof}

The following properties about $X$ will be useful later on. 

\begin{lemma}
\label{lemma:normal-bundle}
Let $E$ be the exceptional divisor of the blow-up $\pi\colon X\to Q$ of the curve~$C_4$.
Then $E\cong\mathbb{F}_n$ for $n\in\{0,2,4,6\}$.
\end{lemma}

\begin{proof}
The curve $C_4$ is rational, so that the $\P^1$-bundle $E$ is isomorphic to a Hirzebruch surface $\F_n$, for some $n\geqslant0$. Let $\mathbf{s}$ be the class of the unique curve on $E$ with $\mathbf{s}^2=-n$. Then $\text{Pic}(E)=\Z \mathbf{f}\oplus\Z \mathbf{s}$, and $\mathbf{s}\cdot\mathbf{f}=1, \,\mathbf{f}^2=0$. \\ 

Now, the divisor $-E|_E$  is linearly equivalent to $a\mathbf{f}+b\mathbf{s}$ for some $a,b\in \Z$, but since $E|_E\cdot\mathbf{f}=E\cdot\mathbf{f}=-1$, we have that $b=1$. Then 
\begin{align*}
E^3&=(-E|_E)^2\\ 
&=2a-n \\
&=-10.
\end{align*}
This gives us that $a=\frac{n-10}{2}$, so that $n$ must be even. Let $H$ be the class of the pullback by $\pi$ of a general hyperplane section of $Q$. Then since $|2H-E|$ is free, $(2H-E)|_E$ is nef, so in particular we have that $(2H-E)|_E\cdot \mathbf{s}\geqslant 0$.\\

We have that $H|_E\sim4\mathbf{f}$ since $\pi(H)\cdot C_4$ consists of 4 points (counted with multiplicity). Thus, we have
\begin{align*}
(2H-E)|_E\cdot\mathbf{s}&=(8+a)\mathbf{f}\cdot \mathbf{s}+\mathbf{s}^2\\
&=8+a-n \\ 
&\geqslant0,
\end{align*}
and so $n\leqslant 6$.
\end{proof}

\begin{lemma}
\label{lemma:del-Pezzo}
\raggedright Let $P$ be a point in $X$ such that $P\notin E\cup E^\prime$, let $S$ be a general surface in $|H|$ containing $P$.
Then
\begin{enumerate}
\item the surface $S$ is a smooth del Pezzo surface of degree $4$,
\item the point $P$ is not contained in any $(-1)$-curve in $S$.
\end{enumerate}
\end{lemma}

\begin{proof}
First we prove (1). Let $\ol{S}=\pi(S)$ and $\ol{P}=\pi(P)$. Then $\ol{S}$ is a general hyperplane section of the quadric $Q$ that contains the point $\ol{P}$. This implies that $\ol{S}$ is a smooth quadric surface, so $\ol{S}\cong\P^1\times\P^1$. Then $\ol{S}$ intersects $C_4$ in 4 distinct points, so that $S$ is the blow-up of $\P^1\times\P^1$ in four points, $O_1,O_2,O_3,O_4\in\ol{S} $. Then $S$ is a smooth del Pezzo surface of degree $4$ if and only if the anti-canonical divisor $-K_S$ is ample, which is the case if and only if: \\
    
    \begin{enumerate}
        \item[(a)] No two of the $O_i$ lie on any $(0,1)$ or $(1,0)$-divisor of $\ol{S}$. \\
        
        Such divisors are lines in $\ol{S}$, and $O_i,O_j$ lie on a line in $\ol{S}$, which is just the secant through $O_i,O_j$, if and only if this secant lies in $Q$. Now the hyperplane sections of $Q$ which contain $\ol{P}$ are parametrised by $\P^3$. Recall that the secants of $C_4$ which are contained in $Q$ span a surface $V_3\cap Q=\pi(E')$, and that $E'$ is a $\P^1$-bundle over the curve $C_4'$, so that such secants are parametrised by $C_4'$. We define the incidence subvariety
        $$I=\big\{\,(H_{\ol{P}},l)\,\,\mid\,\, l\subset H_{\ol{P}}\,\big\}\subset \P^3\times C_4'.$$
        Let $\text{pr}_1\colon I\to\P^3,\text{pr}_2\colon I\to C_4'$ be the natural projections. Then since the fibres of $\text{pr}_2$ are hyperplanes in $\P^4$ which contain $\ol{P}$ and a given secant, and these are parametrised by $\P^1$, we see that $I$ is a $\P^1$-bundle over a curve. In particular, $I$ a surface. 
        
        Now, the variety $\text{pr}_1(I)$ parametrises hyperplane sections of $Q$ which contain a secant of $C_4'$, but since $\text{dim}\,\text{pr}_1(I)\leqslant2$, a general hyperplane section of $Q$ does not contain any secant of $C_4$.\\
        
        \item[(b)] $O_1,O_2,O_3,O_4$ don't lie on a $(1,1)$-divisor of $\ol{S}$. \\

        This is clear as such divisors are given by plane sections of $\ol{S}$ (considered as a smooth quadric in $\P^3$), so that $O_1,O_2,O_3,O_4$ lie on a $(1,1)$-divisor of $\ol{S}$ if and only if they lie on the intersection of two distinct hyperplanes in $\P^4$. This is impossible because four distinct points of $C_4$ lie on a unique hyperplane. \\ 
        
    \end{enumerate}

Since $S$ is a del Pezzo surface of degree 4, it contains 16 $(-1)$-curves. They are: \\ 
    
    \begin{itemize}
        \item The four exceptional curves $E_i$, \\
        \item The strict transforms of the eight lines $L_i,L_i'\subset\ol{S}$ through the $O_i$ (that is, $L_i$ and $L_i'$, are the two rulings of $\ol{S}$ through $O_i$), \\

        \item The strict transforms of the $(1,1)$-curves $C_i\subset\ol{S}$ through $O_j,O_k,O_l$. \\ 
    \end{itemize}

    Now we prove (2). Note that $\ol{P}$ lies on a $l_i$ line in $\ol{S}$ through $O_i$ if and only if $l_i$ lies in $Q$. The union of lines in $Q$ through $\ol{P}$ is given by the intersection $Q_P=Q\cap T_{\ol{P}}Q$, which is a hyperplane section of $Q$ singular at $\ol{P}$, isomorphic to a quadric cone in $\P^3$. Then the curve $C_4$ intersects $Q_P$ in at most 4 points, so there are at most four lines through $\ol{P}$ and $C_4$ which lie in $Q$. The hyperplane section $\ol{S}$ of $Q$ cuts out at most 4 points $O_1,O_2,O_3,O_4$ on $C_4$, and so in general the line $l_i$ through $O_i$ and $\ol{P}$ does not lie in $Q$.\\ 

    Now suppose that $\ol{P}\in C_i$, so that $\ol{P}$ lies in the plane $\Pi$ spanned by $O_j,O_k,O_l$ in $\P^4$. Let $p:\mathbb{P}^4\dasharrow\mathbb{P}^3$ be the projection from the point $\ol{P}$. Recall that all secants of the curve $C_4$ that are contained in $Q$ are contained in the singular threefold $\pi(E')$. Since $P$ is not contained in $E'$, we see that $\ol{P}$ is not contained in any secant line of the curve $C_4$. Thus, $p(C_4)$ is a smooth rational curve of degree $4$ in $\P^3$, $p(\Pi)$ is a trisecant of the curve $p(C_4)$, and $\pi(\overline{S})$ is a general plane in $\mathbb{P}^3$. Recall from \cite[Exercise 6.1]{Hartshorne} that $p(C_4)$ is contained in a unique smooth quadric surface $S_2$, and that all trisecants $p(C_4)$ are contained in $S_2$. Since $p(\overline{S})$ is a general plane, $p(\overline{S})\cap S_2$ is then a smooth conic, which is a contradiction.
\end{proof}

\section{Proof of Classification Theorem}
\label{section:classification}
Recall that any quadric $Q$ containing $C_4$ is given by the homogeneous equation
$$s_0f_0+s_1f_1+s_2f_2+s_3f_3+s_4f_4+s_5f_5=0,$$
where
\begin{center}
\begin{tabular}{ l c }
$f_0=x_3^2-x_2x_4$, & \text{weight 6}\\
$f_1=x_2x_3-x_1x_4$, & \text{weight 5}\\
$f_2=x_2^2-x_0x_4$, & \text{weight 4}\\
$f_3=x_1x_2-x_0x_3$, & \text{weight 3}\\
$f_4=x_1^2-x_0x_2$, & \text{weight 2}\\
$f_5=3x_2^2-4x_1x_3+x_0x_4$ & \text{weight 4}
\end{tabular}
\end{center}

The weights listed are those with respect to the $\C^*$-action described in Remark \ref{remark:GIT}:
$$t\colon [x_0:x_1:x_2:x_3:x_4]\mapsto[x_0:tx_1:t^2x_2:t^3x_3:t^4x_4].$$
on $\P^4$, with respect to which $C_4$ is invariant. We note the Hessian determinant of $Q$:
$$-2(s_0s_4 - s_1s_3 + s_2^2 + 2s_2s_5 - 3s_5^2)(4s_0s_2s_4 - s_0s_3^2 - 4s_0s_4s_5 - s_1^2s_4 + 4s_1s_3s_5 - 16s_2s_5^2 + 16s_5^3).$$

\begin{theorem}
    \label{theorem:classification}
There is $\phi\in\mathrm{SL}_2(\mathbb{C})$ such that one of the following cases holds:
\begin{enumerate}
\item[$(\mathrm{1})$] $\phi^*(Q)=\{\mu(f_0+f_4)+\lambda f_2+f_5=0\}$ for some $\lambda,\mu\in\mathbb{C}$ such that
$$(\mu - 2)(\mu + 2)(\lambda - 1)(\lambda^2 + \mu^2 + 2\lambda - 3)\neq0,$$
\item[$(\mathrm{2})$] $\phi^*(Q)=\{f_0+\lambda f_2+f_5=0\}$ for some $\lambda\in\mathbb{C}\setminus\{1,-3\}$,
\item[$(\mathrm{3})$] $\phi^*(Q)=\{f_1+f_5=0\}$.
\end{enumerate}
\end{theorem}

Let's prove this theorem. Recall from Section 1 that we fixed a faithful $\text{SL}_2(\C)$-action on the $\P^4=\P(\text{Sym}^4(\C^2))$, which is given explicitly by the following injection $\text{SL}_2(\C)\hookrightarrow\text{SL}_5(\C)$:
$$
\begin{pmatrix}a & b \\ c & d\end{pmatrix}\mapsto\begin{pmatrix}
a^4 & 4a^3b & 6a^2b^2 & 4ab^3 & b^4 \\
a^3c & a^3d + 3a^2bc & 3a^2bd + 3ab^2c & 3ab^2d + b^3c & b^3d \\
a^2c^2 & 2a^2cd + 2abc^2 & a^2d^2 + 4abcd + b^2c^2 & 2abd^2 + 2b^2cd & b^2d^2 \\
ac^3 & 3ac^2d + bc^3 & 3acd^2 + 3bc^2d & ad^3 + 3bcd^2 & bd^3 \\
c^4 & 4c^3d & 6c^2d^2 & 4cd^3 & d^4 \\
\end{pmatrix}.
$$
So let $\phi\in \text{SL}_2(\C)$.  Then $\phi^*(Q)$ is given by
$$
s_0'f_0+s_1'f_1+s_2'f_2+s_3'f_3+s_4'f_4+s_5'f_5=0,
$$
for some $[s_0':s_1':s_2':s_3':s_4':s_5']\in \P^5$. \\

Throughout the proof we will frequently use the fact that if two of $s_i,s_j$ are nonzero and correspond to polynomials of different weights, then we can act by $\C^*$ to transform $Q$ into one with $s_i'=s_j'$. We will also use the involution
$$\iota\colon[x_0:x_1:x_2:x_3:x_4]\mapsto[x_4:x_3:x_2:x_1:x_0],$$
which $C_4$ is invariant with respect to. \\ 

We will show that \textit{in general} there exists a $\phi\in\mathrm{SL}_2(\mathbb{C})$ such that $\phi^*(Q)$ is in the form of case (1). Our strategy will be to find $\phi$ such that $s_1'=s_3'=0$. Then if $s_0'=s_4'=0$ or $s_0',s_4'\neq0$, then we can act by $\C^*$ to make $s_0'=s_4'$, so we are in case (1). If on the other hand precisely one of $s_0',s_4'$ is nonzero -- and without loss of generality we may assume it's $s_0'$ -- then we can act by $\C^*$ to make $s_0'=s_5'$, so we get case (2). If such a $\phi$ does not exist, then we show that case (3) holds. In each of these cases, smoothness forces $s_5'\neq0$, so we may assume $s_5'=1$. Thus, the statement is proved.\\ 

The above statement follows from the next three lemmas. \\

If $s_1=s_3=0$ then we're done, so we may assume that $s_1$ is nonzero (remember we can act by $\iota$).

\begin{lemma}
    Suppose that $s_1$ is nonzero. If $s_0=s_2=s_3=s_4=0$, then there exists $\phi\in\mathrm{SL}_2(\mathbb{C})$ such that
    $$\phi^*(Q)=\{f_1+f_5=0\}.$$
    On the other hand if one of $s_0,s_2,s_3,s_4$ is nonzero, then there exists $\phi\in\mathrm{SL}_2(\mathbb{C})$ such that
    $$\phi^*(Q)=\{s_0f_0+s_2f_2+s_3f_3+s_4f_4+s_5f_5=0\}.$$
\end{lemma}

\begin{proof}
If $s_0=s_2=s_3=s_4=0$, then smoothness forces $s_5\neq0$. Thus, we may assume $s_5=1$ and then act by $\C^*$ to transform the equation of $Q$ into $f_1+f_5=0$.

If not, then by can act by 
$$
\begin{pmatrix}
1 & b \\
0 & 1
\end{pmatrix}.$$
Then we have that
$$s_1\mapsto2s_4b^3 + 3s_3b^2 + 4bs_2 + s_1,$$
so that if not all of $s_2,s_3,s_4$ are zero then we can choose any root $b$ of $2s_4b^3 + 3s_3b^2 + 4bs_2 + s_1=0$, so that the equation of $Q$ is transformed into the required form. On the other hand if $s_2=s_3=s_4=0$, then we can act by 
$$
\begin{pmatrix}
1 & 0 \\
c & 1
\end{pmatrix}.
$$
This transforms $s_1\mapsto s_1 + 2cs_0$, and since $s_0\neq0$ by assumption, so we can choose $c=-s_1/2s_0$ to transform the equation of $Q$ into the required form.
\end{proof}
Hence to prove Theorem~\ref{theorem:classification}, we may assume that $s_1=0$. We want to show that there exists $\phi\in\mathrm{SL}_2(\mathbb{C})$ such that $$\phi^*(Q)=\{s_0f_0+s_2f_2+s_4f_4+s_5f_5=0\}.$$
If $s_3=0$ then we're done, so we may assume $s_3\neq 0$. 

\begin{lemma}
\label{lemma:12}
    Suppose that that $s_0=s_1=0$, and $s_3\neq0$. Then there exists $\phi\in\mathrm{SL}_2(\mathbb{C})$ such that $\phi^*(Q)$ is given by one of the following equations:
\begin{itemize}
    \item $f_0+\lambda f_2+ f_5=0$, for $\lambda\notin\{1,-3\}$,  
    \item $f_1+f_5=0$
\end{itemize}
\end{lemma}

\begin{proof}
Note that the Hessian determinant of the equation of $Q$ is 
$$32s_5^2(3s_5 + s_2)(s_2 - s_5)^2,$$
so that whenever the equation satisfies $s_0=s_1=0$ then for $Q$ to be smooth we must have $s_5\neq0$, and $s_2\notin\{ -3s_5,s_5\}$. \\

Since $s_3\neq0$ by assumption, we can act by $\C^*$ and rescale to make $s_3=s_5=1$. We then can act by the involution $\iota$ to put the equation of $Q$ into the form
$$
s_0f_0+f_1+s_2f_2+f_5=0.
$$

Now let us find $\phi$ such that $s_1'=s_3'=s_4'=0.$ We act by the matrix
$$
\begin{pmatrix}
    1 & b \\
    0 & 1
\end{pmatrix},
$$
so that then $s_1'=4bs_2+1$. Then if $s_2\neq0$, a choice of $b=-1/4s_2$ makes $s_1'=s_3'=s_4'=0$, and so $\phi$ puts $Q$ into the form 

$$f_0+\lambda f_2+ f_5=0,$$
which is smooth if and only if $\lambda\notin\{1,-3\}$.\\

 If instead $s_2=0$ then a choice of $b=-s_0/2$ makes $s_0'=0$, and then $\phi$ puts $Q$ into the form:

$$f_1+f_5=0.$$
\end{proof}

Recall that we may assume $s_1=0$, and we are trying to find $\phi\in\text{SL}_2(\C)$ such that the equation of $\phi^*(Q)$ has $s_1'=s_3'=0$.

The case $s_0=0$ has been classified in Lemma~\ref{lemma:12}, so we may assume $s_0\neq0$. Furthermore if $s_3=0$ then we are done, so we may assume $s_3\neq0$. The following lemma completes the proof of Theorem \ref{theorem:classification}.

\begin{lemma}
    Suppose $s_1=0$, and $s_0,s_3\neq0$, and furthermore that the quantity 
    $$
h=256s_2^4s_4 - 128s_2^2s_4^2 + 64s_2^3 + 16s_4^3 - 144s_2s_4 - 27
$$ 
is nonzero. Then there exists $\phi\in\mathrm{SL}_2(\C)$ such that $s_1'=s_3'=0$. Conversely if $h=0$, then there exists $\phi\in\mathrm{SL}_2(\C)$ such that $\phi^*(Q)$ satisfies the conditions of Lemma~\ref{lemma:12}: $s_0'=s_1'=0$ and $s_3'\neq0$.
\end{lemma}

\begin{proof}
Firstly, we can act by $\C^*$ and scale to make $s_0=s_3=1$. The equation of $Q$ is then 
$$f_0+s_2f_2+f_3+s_4f_4+s_5f_5=0.$$
We act by the $\mathrm{SL}_2(\mathbb{C})$ matrix
$$
\begin{pmatrix}
1 & b \\
c & 1+bc
\end{pmatrix},
$$
then we have that
\begin{align*}
    s_1'&=(2c^4 - 8s_2c^2 + 4c + 2s_4)b + 2c^3   - 4s_2c  + 1 \\
    s_3'&=(2c^4 - 8s_2c^2 + 4c + 2s_4)b^3 + (6c^3 - 12cs_2 + 3)b^2 + (6c^2 - 4s_2)b + 2c.
\end{align*}
We then substitute the quantity
$$b=\frac{-(2c^3   - 4s_2c  + 1)}{2c^4 - 8s_2c^2 + 4c + 2s_4}$$
into $s_3'$ to get:
$$
s_3'= \frac{g_1}{g_2^2},
$$
where
\begin{align*}
g_1&=2c^6 + (-16s_2^2 + 4s_4)c^5 + 20c^4s_2 - 10c^3 - 10c^2s_4 + (16s_2^2s_4 - 4s_4^2 + 4s_2)c - 4s_2s_4 - 1, \\ 
g_2&=2c^4 - 8s_2c^2 + 4c + 2s_4.
\end{align*}

Thus, there exists $\phi\in\text{SL}_2(\C)$ such that the equation of $\phi^*(Q)$ has $s_1'=s_3'=0$ if we can find $c\in\C$ such that $g_1=0$ but $g_2\neq0$. Observe that the resultant of $g_1,g_2$ with respect to $c$ is equal to $h^3$. Thus if $h\neq0$, then any root $c$ of $g_1$ will do.\\

Now if $h=0$, we will now show that there exists $\phi\in\text{SL}_2(\C)$ such that the equation of $\phi^*(Q)$ satisfies the conditions of Lemma~\ref{lemma:12}. We act by the matrix
$$
\begin{pmatrix}
1 & b \\
c & 1+bc
\end{pmatrix}.
$$
Then we have that
\begin{align*}
s_3'&=2bc^4 - 8bc^2s_2 + 2c^3 + 4bc + 2bs_4 - 4cs_2 + 1 \\
s_4'&=c^4 - 4c^2s_2 + 2c + s_4.
\end{align*}
The resultant of $s_4'$ and $s_3'-2bs_4'$ with respect to $c$ is
$$256s_2^4s_4 - 128s_2^2s_4^2 + 64s_2^3 + 16s_4^3 - 144s_2s_4 - 27,$$
which is equal to $h$. This was assumed to be zero, so that there exists a common root of these two polynomials. Thus, we can make $s_3'$ and $s_4'$ zero simultaneously, and so by composing the above matrix transformation with $\iota$, we have that $Q$ can be put into the form $s_0'=s_1'=0$. 
\end{proof}

\section{Proof of the Technical lemmas}
\label{section:technical}

We stick to the notations introduced earlier in Section \ref{section:basic-properties}.
\begin{lemma}
\label{lemma:technical1}
Let $P$ be a point in $X$ which is not contained in $ E\cup E^\prime$. Then $\delta_P(X)\geqslant\frac{112}{111}$.
\end{lemma}

\begin{proof}
Let us apply results presented in Section~\ref{section:Abban-Zhuang} to do this.
Let $S$ be a general surface in $|H|$ which contains $P$.
Set
$$
\tau=\mathrm{sup}\Big\{u\in\mathbb{R}_{\geqslant 0}\ \big\vert\ \text{the divisor  $-K_X-uS$ is pseudo-effective}\Big\}.
$$

For $u\leqslant\frac{3}{2}$, let $P(u)$ be the~positive part of the~Zariski decomposition of the~divisor $-K_X-uS$,
and let $N(u)$ be the negative part of the Zariski decomposition of this divisor. Fix $u\in\mathbb{R}_{\geqslant 0}$. Then
$$
-K_X-uS\sim_{\mathbb{R}}(3-u)H-E\sim_{\mathbb{R}}\frac{1}{2}E^\prime+\frac{3-2u}{2}H,
$$
which shows that $\tau=\frac{3}{2}$.

Let $\mathbf{f}$ (resp. $\mathbf{f'}$) be a fibre of the contraction morphism $\pi$ (resp. $\pi'$) over a general point of $C_4$ (resp. $C_4'$). Then the Mori cone $\overline{\text{NE}}(X)$ of $X$ is generated the numerical equivalence classes of $\mathbf{f}$ and $\mathbf{f'}$. Recall from Section~\ref{section:basic-properties} that we have the following intersection numbers on $X$:
$$
    E \cdot \mathbf{f}=-1,\quad    \pi^*D\cdot\mathbf{f}=0,
$$
for all divisors $D\in\text{Pic}(Q).$ So $H\cdot\mathbf{f}=0$ in particular. Then we have:
$$
    (-K_X-uS)\cdot\mathbf{f}=1,\quad    (-K_X-uS)\cdot\mathbf{f'}=1-u.
$$
So $-K_X-uS$ is nef for $0\leqslant u \leqslant 1$. Thus, we have
$$
P(u)\sim_{\R}
\left\{\aligned
&(3-u)H-E \ \text{if $0\leqslant u\leqslant 1$}, \\
&(3-2u)(2H-E)\ \text{if $1\leqslant u\leqslant \frac{3}{2}$},
\endaligned
\right.
$$
and
$$
N(u)=\left\{\aligned
&0 \ \text{if $0\leqslant u\leqslant 1$}, \\
&(u-1)E^\prime\ \text{if $1\leqslant u\leqslant \frac{3}{2}$},
\endaligned
\right.
$$
and so:
\begin{multline*}
S_X(S)=\frac{1}{28}\int_{0}^{\frac{3}{2}}\big(P(u)\big)^3du=\frac{1}{28}\int_{0}^{1}\big((3-u)H-E\big)^3du+\frac{1}{28}\int_{1}^{\frac{3}{2}}\big((3-2u)(2H-E)\big)^3du=\\
=\frac{1}{28}\int_{0}^{1}28-2u^3+18u^2-42udu+\frac{1}{28}\int_{1}^{\frac{3}{2}}2(3-2u)^3du=\frac{51}{112}.\quad\quad\quad\quad\quad
\end{multline*}

By Lemma~\ref{lemma:del-Pezzo}, $S$ is a smooth del Pezzo surface of degree $4$.
Let us identify $S$ with its anti-canonical image.
Thus, we consider $S$ as an intersection of two quadrics in $\mathbb{P}^4$.
Then~$S$ contains $16$ lines, which are $(-1)$-curves in $S$.
It follows from Lemma~\ref{lemma:del-Pezzo} that $P$ is not contained in any of them.

Observe that $\pi(S)\cong\mathbb{P}^1\times\mathbb{P}^1$,
and $\pi$ induces a birational morphism $\pi\vert_{S}\colon S\to\pi(S)$
that blows up the four intersection points $\pi(S)\cap C_4$.
Let $C_1$ and $C_2$ be the strict transforms on $S$ of the two rulings of the surface $\pi(S)$
which pass through the point $\pi(P)$.
Then $C_1$ and $C_2$ are irreducible conics.
On the other hand, the linear systems $|-K_S-C_1|$ and $|-K_S-C_2|$ are basepoint-free pencils.
Let~$Z_1$ be the curve in $|-K_S-C_1|$ containing $P$,
and let $Z_2$ be the curve in $|-K_S-C_2|$ containing~$P$.
Then $Z_1$ and $Z_2$ are irreducible conics (as otherwise $Z_1$ for instance would be a union of two $(-1)$-curves, which would violate part (2) of Lemma~\ref{lemma:del-Pezzo}). Moreover, we have
$$
P(u)\big\vert_{S}\sim_{\mathbb{R}}
\left\{\aligned
&\frac{3-2u}{2}(C_1+C_2)+\frac{1}{2}(Z_1+Z_2)\ \text{if $0\leqslant u\leqslant 1$}, \\
&(3-2u)(C_1+Z_1)\ \text{if $1\leqslant u\leqslant \frac{3}{2}$}.
\endaligned
\right.
$$
The~intersection form of the~curves $C_1$, $C_2$, $Z_1$, $Z_2$ is given in the following table:
\begin{center}\renewcommand{\arraystretch}{1.6}
\begin{tabular}{|c||c|c|c|c|}
    \hline
$\bullet$      &\quad  $C_1$ \quad\quad& \quad$C_2$ \quad\quad & \quad $Z_1$\quad\quad & \quad $Z_2$\quad\quad \\
\hline
\hline
 $C_1$        & $0$  &  $1$ & $2$ & $1$ \\
\hline
 $C_2$        & $1$  & $0$  & $1$ & $2$ \\
\hline
 $Z_1$        & $2$  & $1$  & $0$ & $1$ \\
\hline
 $Z_2$        & $1$  & $2$  & $1$  & $0$\\
\hline
\end{tabular}
\end{center}

Now, let $g\colon\widetilde{S}\to S$ be a blow-up of the~point $P$, let $Y$ be the~exceptional curve of the~blow-up~$g$,
and let $\widetilde{C}_1$, $\widetilde{C}_2$, $\widetilde{Z}_1$, $\widetilde{Z}_2$ be the strict transforms on $\widetilde{S}$ of the~curves
$C_1$, $C_2$, $Z_1$, $Z_2$, respectively.~Set
$$
\widetilde{t}(u)=\sup\Big\{v\in \mathbb{R}_{\geqslant 0} \ \big| \ \text{the $\mathbb{R}$-divisor}\ g^*\big(P(u)|_S\big)-vY \text{ is pseudo-effective}\Big\}
$$
for every $u\leqslant\frac{3}{2}$. Later, we will see that
\begin{equation}
\label{equation:t-u-blow-up}
\widetilde{t}(u)=\left\{\aligned
&4-2u\ \text{if $0\leqslant u\leqslant 1$}, \\
&6-4u\ \text{if $0\leqslant u\leqslant \frac{3}{2}$}.
\endaligned
\right.
\end{equation}
For $v\in [0,\widetilde{t}(u)]$, let $\widetilde{P}(u,v)$ be the~positive part of the~Zariski decomposition of $g^*(P(u)|_S)-vY$,
and let $\widetilde{N}(u,v)$ be the~negative part of this~Zariski decomposition. Then since $P\notin E'$, we have $P\notin\text{Supp}(N(u))$, so that Remark~\ref{remark:6} gives the following simplifications:
$$
S\big(W_{\bullet, \bullet}^{S};Y\big)=\frac{3}{28}\int_0^{\frac{3}{2}}\int_0^{\widetilde{t}(u)}\big(\widetilde{P}(u,v)\big)^2dvdu,
$$
and for every point $O\in Y$, we have
$$
S\big(W_{\bullet, \bullet,\bullet}^{S,Y};O\big)=\frac{3}{28}\int_0^{\frac{3}{2}}\int_0^{\widetilde{t}(u)}\big(\widetilde{P}(u,v)\cdot Y\big)^2dvdu+F_O\big(W_{\bullet,\bullet,\bullet}^{S,Y}\big),
$$
where
$$
F_O\big(W_{\bullet, \bullet,\bullet}^{S,Y}\big)=\frac{6}{28}\int_0^{\frac{3}{2}}\int_0^{\widetilde{t}(u)}\big(\widetilde{P}(u,v)\cdot Y\big)\times\mathrm{ord}_O\big(\widetilde{N}(u,v)\big|_{Y}\big)dvdu.
$$

Then Theorem~\ref{theorem:Hamid-Ziquan-Kento-3}  gives
\begin{equation}
\label{equation:Hamid-Ziquan-Kento}
\delta_P(X)\geqslant \min\left\{\frac{112}{51},\frac{2}{S(W_{\bullet,\bullet}^{S};Y)},\inf_{O\in Y}\frac{1}{S(W_{\bullet, \bullet,\bullet}^{S,Y};O)}\right\}.
\end{equation}
Let us use this to show that $\delta_P(X)\geqslant\frac{112}{111}$. First, let us compute $S(W_{\bullet,\bullet}^{S};Y)$.

The~intersection form of the~curves  $\widetilde{C}_1$, $\widetilde{C}_2$, $\widetilde{Z}_1$, $\widetilde{Z}_2$ on the~surface $\widetilde{S}$
is given in this~table:
\begin{center}\renewcommand{\arraystretch}{1.6}
\begin{tabular}{|c||c|c|c|c|c|}
    \hline
$\bullet$      &\quad  $\widetilde{C}_1$ \quad\quad& \quad$\widetilde{C}_2$ \quad\quad & \quad $\widetilde{Z}_1$\quad\quad & \quad $\widetilde{Z}_2$\quad\quad & \quad $Y$\quad\quad \\
\hline
\hline
$\widetilde{C}_1$       & $-1$  &  $0$ & $1$ & $0$ & $1$\\
\hline
$\widetilde{C}_2$       & $0$  & $-1$  & $0$ & $1$ & $1$\\
\hline
 $\widetilde{Z}_1$      & $1$  & $0$  & $-1$ & $0$ & $1$\\
\hline
$\widetilde{Z}_2$       & $0$  & $1$  & $0$  & $-1$& $1$\\
\hline
$Y$                     & $1$  & $1$  & $1$  & $1$ & $-1$\\
\hline
\end{tabular}
\end{center}
On the~other hand, we have
$$
g^*(P(u)|_S)-vY\sim_{\mathbb{R}}
\left\{\aligned
&\frac{3-2u}{2}(\widetilde{C}_1+\widetilde{C}_2)+\frac{1}{2}(\widetilde{Z}_1+\widetilde{Z}_2)+(4-2u-v)Y\ \text{if $0\leqslant u\leqslant 1$}, \\
&(3-2u)(\widetilde{C}_1+\widetilde{Z}_1)+(6-4u-v)Y\ \text{if $1\leqslant u\leqslant \frac{3}{2}$}.
\endaligned
\right.
$$
This gives \eqref{equation:t-u-blow-up},
because the intersection form of the curves $\widetilde{C}_1$, $\widetilde{C}_2$, $\widetilde{Z}_1$, $\widetilde{Z}_2$ is semi-negative definite.
Now, intersecting  $g^*(P(u)|_S)-vY$ with $\widetilde{C}_1$, $\widetilde{C}_2$, $\widetilde{Z}_1$, $\widetilde{Z}_2$,
we compute $\widetilde{P}(u,v)$ and $\widetilde{N}(u,v)$ for every non-negative real number $u\leqslant\frac{3}{2}$ and every real number $v\in[0,\widetilde{t}(u)]$.
Namely, if $u\leqslant 1$, then
$$
\widetilde{P}(u,v)\sim_{\R}
\left\{\aligned
&\frac{3-2u}{2}(\widetilde{C}_1+\widetilde{C}_2)+\frac{1}{2}(\widetilde{Z}_1+\widetilde{Z}_2)+(4-2u-v)Y\ \text{if $0\leqslant v\leqslant 3-u$}, \\
&\frac{9-4u-2v}{2}(\widetilde{C}_1+\widetilde{C}_2)+\frac{1}{2}(\widetilde{Z}_1+\widetilde{Z}_2)+(4-2u-v)Y  \ \text{if $3-u\leqslant v\leqslant 4-2u$},
\endaligned
\right.
$$
and
$$
\widetilde{N}(u,v)=
\left\{\aligned
&0\ \text{if $0\leqslant v\leqslant 3-u$}, \\
&(v+u-3)(\widetilde{C}_1+\widetilde{C}_2) \ \text{if $3-u\leqslant v\leqslant 4-2u$}.
\endaligned
\right.
$$
This implies that
$$
\big(\widetilde{P}(u,v)\big)^2=
\left\{\aligned
&2u^2-v^2-12u+14\ \text{if $0\leqslant v\leqslant 3-u$}, \\
&(2u+v-4)(2u+v-8)\ \text{if $3-u\leqslant v\leqslant 4-2u$},
\endaligned
\right.
$$
and
$$
\widetilde{P}(u,v)\cdot Y=
\left\{\aligned
&v\ \text{if $0\leqslant v\leqslant 3-u$}, \\
&6-2u-v\ \text{if $3-u\leqslant v\leqslant 4-2u$}.
\endaligned
\right.
$$
If $1\leqslant u\leqslant\frac{3}{2}$ and $v\in[0,6-4u]$, then $\widetilde{P}(u,v)=(3-2u)(\widetilde{C}_1+\widetilde{Z}_1)+(6-4u-v)Y$  and $\widetilde{N}(u,v)=0$,
so
\begin{align*}
\big(\widetilde{P}(u,v)\big)^2&=(4u-6+v)(4u-6-v),\\
\widetilde{P}(u,v)\cdot Y&=v.
\end{align*}
This gives
\begin{align*}
S\big(W_{\bullet,\bullet}^{S};Y\big)=\frac{3}{28}\int_0^1\int_0^{3-u}2u^2-v^2-12u+14dvdu+\quad\quad\quad\quad\quad\quad\quad\quad\\
+\frac{3}{28}\int_0^1\int_{3-u}^{4-2u}(2u+v-4)(2u+v-8)dvdu+\quad\quad\quad\quad\\
\quad\quad\quad\quad\quad\quad\quad\quad\quad+\frac{3}{28}\int_1^{\frac{3}{2}}\int_{0}^{6-4u}(4u-6+v)(4u-6-v)dvdu=\frac{111}{56}.
\end{align*}
Hence, using \eqref{equation:Hamid-Ziquan-Kento}, we see that $\delta_P(X)\geqslant\frac{112}{111}$
provided that $S(W_{\bullet, \bullet,\bullet}^{S,Y};O)\leqslant\frac{111}{112}$ for every point $O\in Y$.

Fix $O\in Y$. Let us show that $S(W_{\bullet, \bullet,\bullet}^{S,Y};O)\leqslant\frac{111}{112}$.
We have
\begin{multline*}
S\big(W_{\bullet, \bullet,\bullet}^{S,Y};O\big)=\frac{3}{28}\int_0^1\int_0^{3-u}v^2dvdu+\frac{3}{28}\int_0^1\int_{3-u}^{4-2u}(6-2u-v)^2dvdu+\\
+\frac{3}{28}\int_1^{\frac{3}{2}}\int_{0}^{6-4u}v^2dvdu+F_O\big(W_{\bullet,\bullet,\bullet}^{S,Y}\big)=\frac{51}{56}+F_O\big(W_{\bullet,\bullet,\bullet}^{S,Y}\big).\quad\quad\quad\quad\quad\quad
\end{multline*}
Recall that $\widetilde{N}(u,v)=(v+u-3)(\widetilde{C}_1+\widetilde{C}_2)$ for $u\leqslant1$ and $3-u\leqslant v\leqslant4-2u$, and it is zero for all other values of $u,v$. Moreover, we have

$$
\mathrm{ord}_O\big(\widetilde{N}(u,v)\big|_{Y}\big)=
\left\{\aligned
&0\ \text{if $O\notin\widetilde{C}_1\cup\widetilde{C}_2$}, \\
&(v+u-3) \ \text{if $O\in\widetilde{C}_1\cup\widetilde{C}_2$}.
\endaligned
\right.
$$

Thus if $O\notin\widetilde{C}_1\cup\widetilde{C}_2$ then $F_O(W_{\bullet,\bullet,\bullet}^{S,Y})=0$,
so that  $S(W_{\bullet, \bullet,\bullet}^{S,Y};O)=\frac{51}{56}<\frac{111}{112}$ as required.
On the other hand if $O\in\widetilde{C}_1\cup\widetilde{C}_2$, then
$$
F_O\big(W_{\bullet, \bullet,\bullet}^{S,Y}\big)=\frac{6}{28}\int_0^{1}\int_{3-u}^{4-2u}(6-2u-v)(v+u-3)dvdu=\frac{9}{112},
$$
so $S(W_{\bullet, \bullet,\bullet}^{S,Y};O)=\frac{111}{112}$.
Thus, applying \eqref{equation:Hamid-Ziquan-Kento}, we get $\delta_P(X)\geqslant\frac{112}{111}$ as claimed.
\end{proof}
\begin{lemma}
\label{lemma:technical2}
Let $D$ be a prime divisor over $X$ such that its centre on $X$ is a curve $C\subset E\cup E'$, and suppose that $C$ is not a fibre of $\pi|_E$ or $\pi'|_{E'}$. Then $A_X(D)>S_X(D)$. 
\end{lemma}

\begin{proof}
Without loss of generality, we may assume that $C\subset E$.
Let us apply Theorem~\ref{theorem:Hamid-Ziquan-Kento-1} with~$S=E$.
To do this, we fix $u\in\mathbb{R}_{\geqslant 0}$. Then
$$
-K_X-uE\sim_{\mathbb{R}}3H-(1+u)E\sim_{\mathbb{R}} E^\prime+(1-u)E.
$$

This shows that $-K_X-uE$ is pseudoeffective $\iff$ $u\leqslant 1$.

With $\mathbf{f},\mathbf{f'}$ as before, we have that 

$$
(-K_X-uE)\cdot\mathbf{f} =1+u,\quad (-K_X-uE)\cdot\mathbf{f'} =1-2u,
$$
\\
so that $-K_X-uE$ is nef for $0\leqslant u\leqslant\frac{1}{2}$. For $u\leqslant 1$, let $P(u)$ be the positive part of the~Zariski decomposition of the~divisor $-K_X-uE$,
and let $N(u)$ be its negative part. Then
$$
P(u)\sim
\left\{\aligned
&3H-(1+u)E \ \text{if $0\leqslant u\leqslant \frac{1}{2}$}, \\
&(3-3u)(2H-E)\ \text{if $\frac{1}{2}\leqslant u\leqslant 1$},
\endaligned
\right.
$$
and
$$
N(u)=\left\{\aligned
&0 \ \text{if $0\leqslant u\leqslant \frac{1}{2}$}, \\
&(2u-1)E^\prime\ \text{if $\frac{1}{2}\leqslant u\leqslant 1$}.
\endaligned
\right.
$$
Thus, we have
\begin{multline*}
S_X(E)=\frac{1}{28}\int_{0}^{1}\big(P(u)\big)^3du=\frac{1}{28}\int_{0}^{\frac{1}{2}}\big(3H-(1+u)E\big)^3du+\frac{1}{28}\int_{\frac{1}{2}}^{1}\big((3-3u)(2H-E)\big)^3du=\\
=\frac{1}{28}\int_{0}^{\frac{1}{2}}10u^3-6u^2-42u+28du+\frac{1}{28}\int_{\frac{1}{2}}^{1}54(1-u)^3du=\frac{19}{56}.\quad\quad\quad\quad\quad
\end{multline*}

By Lemma~\ref{lemma:normal-bundle}, $E\cong\mathbb{F}_{n}$ for some $n\in\{0,2,4,6\}$. Let $\mathbf{s}$ be the unique section of the~projection $E\to C_4$ such that $\mathbf{s}^2=-n$,
and let $\mathbf{f}$ be the fibre of this projection. Then, arguing as in the proof of Lemma~\ref{lemma:normal-bundle},
we see that 
$$
P(u)\big\vert_{E}\sim_{\mathbb{R}}
\left\{\aligned
&(u+1)\mathbf{s}+\frac{n+14-(10-n)u}{2}\mathbf{f}\ \text{if $0\leqslant u\leqslant \frac{1}{2}$}, \\
&(3-3u)\mathbf{s}+\frac{3(n+6)(1-u)}{2}\mathbf{f}\ \text{if $\frac{1}{2}\leqslant u\leqslant 1$}.
\endaligned
\right.
$$
In particular, if $\frac{1}{2}\leqslant u\leqslant 1$, then
$$
\big(P(u)\big\vert_{E}\big)^2=54(1-u)^2.
$$
On the other hand, it follows from Theorem~\ref{theorem:Hamid-Ziquan-Kento-1} that
\begin{equation}
\label{equation:AZF-E}
\frac{A_X(D)}{S_X(D)}\geqslant\min\Bigg\{\frac{56}{19},\frac{1}{S\big(W^E_{\bullet,\bullet};C\big)}\Bigg\},
\end{equation}
where 
$$
S\big(W^E_{\bullet,\bullet};C\big)=\frac{3}{28}\int_{\frac{1}{2}}^1\big(P(u)\big\vert_{E}\big)^2(2u-1)\mathrm{ord}_{C}\big(E^\prime\big\vert_{E}\big)du+\frac{3}{28}\int_0^1\int_0^{\infty}\mathrm{vol}\big(P(u)\big\vert_{E}-vC\big)dvdu.
$$
Thus, to complete the proof, it is enough to show that $S(W^E_{\bullet,\bullet};C)<1$. Let us do this.

Observe that 
$$
E^\prime\big\vert_{E}\sim\big(3H-2E\big)\big\vert_{E}\sim2\mathbf{s}+(2+n)\mathbf{f}. 
$$
Thus, we have $\mathrm{ord}_{C}(E^\prime\vert_{E})\leqslant 2$. This gives 
$$
\frac{3}{28}\int_{\frac{1}{2}}^1\big(P(u)\big\vert_{E}\big)^2(2u-1)\mathrm{ord}_{C}\big(E^\prime\big\vert_{E}\big)du
\leqslant\frac{3}{28}\int_{\frac{1}{2}}^{1}108(2u-1)(1-u)^2du=\frac{27}{224}.
$$
Moreover, since $C$ is not a fibre of the projection $E\to C_4$ we may write $C\sim \mathbf{s}+\mathbf{e}$, where $\mathbf{e}$ is an effective divisor on $E$. Then since volume of real classes is non-decreasing in effective directions (see for example \cite[Example 2.2.48]{positivity}), we have that $\mathrm{vol}\big(P(u)\big\vert_{E}-vC\big)\leq\mathrm{vol}\big(P(u)\big\vert_{E}-v\mathbf{s}\big)$. This gives
$$
\frac{3}{28}\int_0^1\int_0^{\infty}\mathrm{vol}\big(P(u)\big\vert_{E}-vC\big)dvdu\leqslant\frac{3}{28}\int_0^1\int_0^{\infty}\mathrm{vol}\big(P(u)\big\vert_{E}-v\mathbf{s}\big)dvdu.
$$
To estimate the integral in the left hand side of this inequality, set
$$
t(u)=\left\{\aligned
&1+u\ \text{if $0\leqslant u\leqslant \frac{1}{2}$}, \\
&3-3u\ \text{if $\frac{1}{2}\leqslant u\leqslant 1$}.
\endaligned
\right.
$$
Observe that
$$
P(u)\big\vert_{E}-v\mathbf{s}\sim_{\mathbb{R}}
\left\{\aligned
&(u-v+1)\mathbf{s}+\frac{n+14-(10-n)u}{2}\mathbf{f}\ \text{if $0\leqslant u\leqslant \frac{1}{2}$}, \\
&(3-3u-v)\mathbf{s}+\frac{3(n+6)(1-u)}{2}\mathbf{f}\ \text{if $\frac{1}{2}\leqslant u\leqslant 1$}.
\endaligned
\right.
$$
Then $P(u)\big\vert_{E}-v\mathbf{s}$ is pseudoeffective $\iff$ $P(u)\big\vert_{E}-v\mathbf{s}$ is nef $\iff$ $v\leqslant t(u)$.
This gives
\begin{multline*}
\frac{3}{28}\int_0^1\int_0^{\infty}\mathrm{vol}\big(P(u)\big\vert_{E}-v\mathbf{s}\big)dvdu=\frac{3}{28}\int_0^1\int_0^{t(u)}\big(P(u)\big\vert_{E}-v\mathbf{s}\big)^2dvdu=\\
=\frac{3}{28}\int_0^{\frac{1}{2}}\int_0^{1+u}\big((u-v+1)\mathbf{s}+\frac{n+14-(10-n)u}{2}\mathbf{f}\big)^2dvdu+\\
+\frac{3}{28}\int_{\frac{1}{2}}^{1}\int_0^{3-3u}\big((3-3u-v)\mathbf{s}+\frac{3(n+6)(1-u)}{2}\mathbf{f}\big)^2dvdu=\\
=\frac{3}{28}\int_0^{\frac{1}{2}}\int_0^{1+u}14+4u+(n-14)v-10u^2-nv^2+(n+10)vudvdu+\\
+\frac{3}{28}\int_{\frac{1}{2}}^{1}\int_0^{3-3u}54-108u+(3n-18)v+54u^2-v^2n+(18-3n)vudvdu=\frac{23n+546}{896}.
\end{multline*}
In summary therefore, we get 
$$
S(W^E_{\bullet,\bullet};C)\leqslant\frac{27}{224}+\frac{23n+546}{896}=\frac{654+23n}{896}\leqslant\frac{99}{112},
$$
because $n\in\{0,2,4,6\}$. Now, using \eqref{equation:AZF-E}, we get $A_X(D)>S_X(D)$.
\end{proof}

Now, we are ready to prove the Main Theorem.

\section{Proof of Main Theorem}
\label{section:proof}

\begin{theorem}
    \label{theorem:main}
    Suppose $\mathrm{Aut}(Q,C_4)$ contains a subgroup $G\cong\Z_2^2$. Then $X$ is K-polystable.
\end{theorem}

\begin{proof}
    Fix a $G$-equivariant birational morphism $h\colon\widetilde{X}\to X$ and $G$-invariant prime divisor $D\subset \widetilde{X}$. Then $Z=h(D)$ is either a surface, curve or point. We show that $A_X(D)>S_X(D)$ in all possible cases, so that $X$ is K-polystable by \cite[Theorem 1.6]{Fujita2019}. \\
    \begin{itemize}
        \item[\textbf{Case 1.}] $Z$ is a surface. \\ 

        By \cite[Theorem 10.1]{divisorial} $X$ is divisorially stable, so that $$\frac{A_X(D)}{S_X(D)}=\frac{A_X(Z)}{S_X(Z)}>1.$$ 
        \vspace{0.1cm}        
        \item[\textbf{Case 2.}] $Z$ is contained in $E\cup E'$ as either a point or a fibre of $\pi$ or $\pi'$. \\

         Without loss of generality we may assume that $Z$ is contained in $E$. Then $\pi(Z)$ is a $G$-invariant point in $C_4$. Since $C_4$ is rational normal curve $\mathrm{Aut}(Q,C_4)$, and hence $G$, acts faithfully on $C_4$. I claim that that $G$ acts freely on $C_4$, which gives a contradiction. Indeed, if we write $g,h$ for the generators of $G\cong\Z_2\times\Z_2$ then $g$, being finite order, fixes precisely two distinct points $p,q\in C_4$. Then $h(p)=h(g(p))=g(h(p))$, which implies that $h(p)$, and similarly $h(q)$, is a fixed point of $g$. So if $h(p)=p$ and $h(q)=q$, then since $g$ and $h$ are both order 2, they must be equal. Thus $h$ swaps $p$ and $q$, so that $G$ acts freely. \\
            
        \item[\textbf{Case 3.}] $Z$ is a curve contained in $E\cup E'$, and not a fibre of $\pi$ nor $\pi'$.\\

        Then $A_X(D)>S_X(D)$ by Lemma \ref{lemma:technical2}. \\ 

        \item[\textbf{Case 4.}] $Z$ is either a curve or a point, and is not contained in $E\cup E'$. \\ 

        Then there exists a point $P\in Z$ such that $P\notin E\cup E'$. Then $\delta_P(X)>1$ by Lemma \ref{lemma:technical1}, so that
        $$\frac{A_X(D)}{S_X(D)}\geq\delta_P(X)>1.$$
    \end{itemize}
    \noindent Thus $A_X(D)>S_X(D)$ for all $G$-invariant prime divisors $D$ over $X$.
\end{proof}

\section*{Appendix}
\label{section:normal bundles}
In this section, we present a classification of the normal bundle $\mathcal{N}_{C_4}/Q$ of a rational normal quartic curve $C_4$ in a smooth quadric threefold $Q$. A corollary to this is a refinement of the statement of Lemma \ref{lemma:normal-bundle}.\\

Recall that the curve $C_4$ is given by the Veronese embedding:
\begin{align*}
    \nu\colon\P^1&\to\P^4 \\ 
    [u:v]&\mapsto[u^4:u^3v:u^2v^2:uv^3:v^4].
\end{align*}
By Theorem \ref{theorem:classification} we may assume that the quadric $Q$ is given by one of the following equations:

\begin{enumerate}
\item[$(\mathrm{1})$] $\mu(f_0+f_4)+\lambda f_2+f_5=0$, for some $\lambda,\mu\in\mathbb{C}$ such that
$$(\mu - 2)(\mu + 2)(\lambda - 1)(\lambda^2 + \mu^2 + 2\lambda - 3)\neq0,$$
\item[$(\mathrm{2})$] $f_0+\lambda f_2+f_5=0$, for some $\lambda\in\mathbb{C}\setminus\{1,-3\}$,
\item[$(\mathrm{3})$] $f_1+f_5=0$.
\end{enumerate}
Then we have the following classification of the normal bundle $\mathcal{N}_{C_4/Q}$:

\begin{theorem}
\label{normal bundles}
    For a general smooth quadric $Q$ which contains the curve $C_4$, the normal bundle $\mathcal{N}_{C_4/Q}$ is isomorphic to $\O_{\P^1}(5)^{\oplus2}$. Otherwise $\mathcal{N}_{C_4/Q}\cong\O_{\P^1}(4)\oplus\O_{\P^1}(6)$, which holds for the following exceptions:
    \begin{itemize}
        \item In case (1) we have $(\lambda-1)^2-\mu^2=0$,
        \item If $\lambda=-1$ in any of the above cases.
    \end{itemize}
\end{theorem}

\begin{proof}
Observe that $\mathcal{N}_{C_4/Q}$ fits into the following exact sequence 
\begin{equation}
    \label{normals}
    0\longrightarrow\mathcal{N}_{C_4/Q}\longrightarrow\mathcal{N}_{C_4/\P^4}\overset{\psi_f}{\longrightarrow}\mathcal{N}_{Q/\P^4}|_{C_4}\longrightarrow0,
\end{equation}
where the morphism $\psi_f$ depends on the equation $f$ of the quadric $Q$. Clearly $\mathcal{N}_{Q/\P^4}|_{C_4}\cong\O_{\P^1}(8)$. Let us compute $\mathcal{N}_{C_4/\P^4}$. For ease of calculation, we compute the conormal bundle $\mathcal{N}_{C_4/\P^4}^\vee$, and take its dual. For this, consider the following diagram consisting of the Euler sequences for $\P^1$ and $\P^4$, and the conormal exact sequence of $C\subset\P^4$:
\begin{center}
\begin{tikzcd}
                &       0 \arrow[d] & 0 \arrow[d] & & \\
                &     \mathcal{N}_{C_4/\P^4}^\vee \arrow[r, equals, shorten <=15pt, shorten >=15pt] \arrow[d]         & \mathcal{N}_{C_4/\P^4}^\vee \arrow[d]                &       & \\
    0 \arrow[r] & \Omega^1_{\P^4}|_{C_4} \arrow[r] \arrow[d, "(d\nu)^*"] & \O_{\P^4}(-1)|_{C_4}^{\oplus5} \arrow[r] \arrow[d, "(J\nu)^*"] & \O_{\P^1} \arrow[r] \arrow[d, "\cdot 4"] & 0\\
    0 \arrow[r] & \Omega^1_{\P^1} \arrow[r]   \arrow[d]        & \O_{\P^1}(-1)^{\oplus2} \arrow[r] \arrow[d]   & \O_{\P^1} \arrow[r]  & 0 \\
    & 0 & 0 & &
\end{tikzcd}
\end{center}
Then the rows and columns are exact, and a simple diagram chase shows that the conormal bundle $\mathcal{N}^\vee_{C_4/\P^4}$ is isomorphic to the kernel of morphism of sheaves:

$$\O_{\P^1}(-4)^{\oplus5}\to\O_{\P^1}(-1)^{\oplus2}$$
given by the dual of the Jacobian matrix of $\nu$:

$$(J\nu)^*=\begin{pmatrix}
    4u^3 & 3u^2v & 2uv^2 & v^3 & 0 \\ 
    0 & u^3 & 2u^2v & 3uv^2 & 4v^3
\end{pmatrix}.$$
Let us compute $\mathrm{ker}\,(J\nu)^*$. Let $(F_0,F_1,F_2,F_3,F_4)$ be a (local) section of the sheaf $\mathrm{ker}\,(J\nu)^*$. Then we get the two following equations:
\begin{equation}
    \label{eqs1}
\begin{cases}
4u^3F_0+3u^2vF_1+2uv^2F_2+v^3F_3 = 0, \\ 
u^3F_1+2u^2vF_2+3uv^2F_3+4v^3F_4 = 0,
\end{cases}
\end{equation}
from which we deduce that $v$ divides $F_1$ and $u$ divides $F_3$. Thus we can write $F_1=vF_1'$ and $F_3=uF_3'$, for forms $F_1',F_3'$ of degree $-5$. Substituting these into (\ref{eqs1}), we deduce that $v^2$ divides $F_0$ and $u^2$ divides $F_4$. Thus we write $F_0=v^2F_0'$ and $F_4=u^2F_4'$, for forms $F_0',F_4'$ of degree $-6$. Then we get the equations
$$
\begin{cases}
    4u^2F_0'+3uF_1'+2F_2+vF_3' = 0 \\
    uF_1'+2F_2+3vF_3'+4v^2F_4' = 0.
\end{cases}
$$
If subtract the second equation from the first, then we get 
$$
2u^2F_0'+uF_1'-vF_3'-2v^2F_4'=0,
$$
from which we deduce that $v$ divides $2uF_0'+F_1'$, and so there is a form $G$ of degree $-6$ such that $F_1'=vG-2uF_0'$. It then follows that $F_3'=uG-2vF_4'$ and $F_2=u^2F_0'-2uvG+v^2F_4'$. Thus, we have the following kernel relations:
$$\begin{pmatrix}
v^2 & 0 & 0 \\ 
-2uv & v^2 & 0 \\ 
u^2 & -2uv & v^2 \\ 
0 & u^2 & -2uv \\ 
0 & 0 & u^2
\end{pmatrix}\begin{pmatrix} F_0' \\ G \\ F_4'\end{pmatrix}=\begin{pmatrix}F_0 \\ F_1 \\ F_2 \\ F_3 \\ F_4\end{pmatrix}.$$
The above matrix clearly has rank 3, so defines an isomorphism $\O_{\P^1}(-6)^{\oplus3}\xrightarrow{\sim}\mathcal{N}^\vee_{C_4/\P^4}$. Thus we have that
$$\mathcal{N}_{C/\P^4}\cong\O_{\P^1}(6)^{\oplus3}.$$

Using this result and the exact sequence (\ref{normals}), we can compute $\mathcal{N}_{C_4/Q}$. In the notion used in \cite[13]{Izzet}, the ideal of $C_4$ in $\P^4$ is generated by the following determinental forms:
\begin{align*}
    q_{12}&=f_4=x_1^2-x_0x_2, \\ 
    q_{13}&=f_3=x_1x_2-x_0x_3, \\ 
    q_{14}&=\frac{3f_2-f_5}{4}=x_1x_3-x_0x_4, \\ 
    q_{23}&=\frac{f_2+f_5}{4}=x_2^2-x_1x_3, \\ 
    q_{24}&=f_1=x_2x_3-x_1x_4, \\ 
    q_{34}&=f_0=x_3^2-x_2x_4.
\end{align*}
Let the equation of $Q$ be:
$$f=a_{12}q_{12}+a_{13}q_{13}+a_{14}q_{14}+a_{23}q_{23}+a_{24}q_{24}+a_{34}q_{34}=0.$$
Then by Theorem \ref{theorem:classification}, after a projective transformation $Q$ is given by one of the following equations:
\begin{enumerate}[$(1)$]
\item $\mu(q_{12}+q_{34})+(\lambda-1)q_{14}+(\lambda+3)q_{23}=0$, for some $\lambda\in\C\setminus\{1\}$ and $\mu\in\C\setminus\{2,-2\}$ such that $\mu^2\neq-\lambda^2-2\lambda+3$,
\item $q_{34}+(\lambda-1)q_{14}+(\lambda+3)q_{23}=0$, for some $\lambda\in\mathbb{C}\setminus\{1,-3\}$,
\item $q_{24}+3q_{23}-q_{14}=0$.
\end{enumerate}

Let us compute the normal bundle $\mathcal{N}_{C_4/Q}$ of $C_4$ in $Q$ in each of these cases. \\ 

By (\ref{normals}), $\mathcal{N}_{C_4/Q}$ is isomorphic to the kernel of the map 
$$\psi_f\colon\mathcal{N}_{C_4/\P^4}\cong\O_{\P^1}(6)^{\oplus3}\longrightarrow\mathcal{N}_{Q/\P^4}|_{C_4}\cong\O_{\P^1}(8).$$
By \cite[13]{Izzet}, this map is given by the following 1-by-3 matrix:
$$
\begin{pmatrix}
    a_{12}u^2+a_{13}uv+a_{14}v^2 \\
    a_{13}u^2+(a_{14}+a_{23})uv+a_{24}v^2 \\
    a_{14}u^2+a_{24}uv+a_{34}v^2
\end{pmatrix}^\text{T}.
$$ 

\subsection*{Case (1)}
Then the morphism $\psi_f$ is given by the matrix:
$$\begin{pmatrix}\mu u^2+(\lambda-1)v^2 & (2\lambda+2)uv & (\lambda-1)u^2+\mu v^2\end{pmatrix}.$$ Suppose that $(F,G,H)$ is any section of $\O_{\P^1}(6)^{\oplus3}$ such that $\psi_f((F,G,H))=0$. Then:
\begin{equation}
\label{kernel}
    (\mu F + (\lambda - 1) H) u^2 + (2 \lambda + 2) u v G + ((\lambda - 1) F + \mu H) v^2=0.
\end{equation}
Let us assume that $(\lambda-1)^2-\mu^2\neq0$, and also that $\lambda\neq-1$. Then (\ref{kernel}) implies that $u$ divides $(\lambda-1)F+\mu H$ and $v$ divides $\mu F+(\lambda-1)H$. Thus, we obtain the linear equations:
$$
\begin{pmatrix}
    \lambda-1 & \mu \\
    \mu & \lambda-1 
\end{pmatrix}
\begin{pmatrix}
    F \\ H
\end{pmatrix}=
\begin{pmatrix}
    uF' \\ vH'
\end{pmatrix}
$$
for some degree 5 forms $(F',H').$ Since $(\lambda-1)^2-\mu^2\neq0, $ we can solve the above system for $F$ and $H$:
$$
\begin{pmatrix}
    F \\ H
\end{pmatrix}
=\frac{1}{(\lambda-1)^2-\mu^2}
\begin{pmatrix}
    (\lambda-1)uF'-\mu v H' \\ 
    -\mu u F'+(\lambda-1)vH'.
\end{pmatrix}
$$
Then by substituting the above expressions for $F$ and $H$ into (\ref{kernel}), we obtain the kernel relation:
$$\begin{pmatrix}
    \frac{(\lambda-1)u}{(\lambda-1)^2-\mu^2} & -\frac{\mu v}{(\lambda-1)^2-\mu^2} \\ 
    -\frac{v}{2\lambda+2} & \frac{u}{2\lambda+2} \\ 
    -\frac{\mu u}{(\lambda-1)^2-\mu^2} & \frac{(\lambda-1)v}{(\lambda-1)^2-\mu^2}
\end{pmatrix}
\begin{pmatrix}F' \\ H'\end{pmatrix}
=\begin{pmatrix}F \\ G \\ H\end{pmatrix}.$$
The matrix above is rank 2, so that $\mathcal{N}_{C_4/Q}\cong\O_{\P^1}(5)^{\oplus2}.$  

If on the other hand $\lambda=-1$, then instead we see that $u^2$ divides $-2F+\mu H$ and $v^2$ divides $\mu F-2H$, so we obtain the following linear equations:

$$
\begin{pmatrix}
    -2 & \mu \\ \mu & -2 
\end{pmatrix}
\begin{pmatrix}
    F \\ H
\end{pmatrix}=
\begin{pmatrix}
    u^2F'' \\ v^2 H''
\end{pmatrix},
$$
for some degree 4 forms $(F'',H'')$. Since $\mu^2\neq4$ by smoothness, these equations have the unique solution:
$$
\begin{pmatrix}
    F \\ H
\end{pmatrix}=\frac{1}{\mu^2-4}
\begin{pmatrix}
    2u^2F''+\mu v^2 H'' \\ 
    \mu u^2 F''+2v^2 H''
\end{pmatrix}
$$

Substituting these back into (\ref{kernel}) and cancelling $u^2v^2$ gives $H''=-F''$. Thus, we obtain the kernel relations:
$$\begin{pmatrix}
    \frac{2u^2-\mu v^2}{\mu^2-4} & 0 \\ 
    0 & 1 \\ 
    \frac{\mu u^2 -  2v^2}{\mu^2-4} & 0
\end{pmatrix}
\begin{pmatrix}F'' \\ G\end{pmatrix}
=\begin{pmatrix}F \\ G \\ H\end{pmatrix},$$
so that $\mathcal{N}_{C_4/Q}\cong\O_{\P^1}(4)\oplus\O_{\P^1}(6).$ \\

Now if $(\lambda-1)^2-\mu^2=0$, then $\lambda=1+\alpha\mu$, where $\alpha\in\{1,-1\}$. Then we have the following expression:
\begin{equation}
    \label{ker4}
    \mu(F+\alpha H)(u^2+\alpha v^2)+(4+2 \alpha \mu)uvG=0.
\end{equation}
Then we see that by (\ref{ker4}) $u^2+\alpha v^2$ divides $G$ ($\mu\neq0$ since $Q$ is smooth). Thus, we have that $(u^2+\alpha v^2)G'=G$ for some degree 4 form $G'$, which gives the following kernel relation:

$$
\begin{pmatrix}
1 & 0 \\ 0 & u^2+\alpha v^2 \\ -\alpha & -(4\alpha\mu^{-1}+2)uv
\end{pmatrix}
\begin{pmatrix}
F \\ G'
\end{pmatrix}
=\begin{pmatrix}F \\ G \\ H \end{pmatrix}.
$$
Thus $\mathcal{N}_{C_4/Q}\cong\O_{\P^1}(4)\oplus\O_{\P^1}(6).$ \\ 

\subsection*{Case (2)} Then the morphism $\psi_f$ is given by the matrix: 
$$\begin{pmatrix}(\lambda-1) v^2 & (2\lambda+2)uv & (\lambda-1)u^2+v^2\end{pmatrix}.$$ First assume that $\lambda\neq-1$. Suppose that $(F,G,H)\in\mathrm{ker}\,\psi_f$. Then we have the relation:
\begin{equation}
    \label{ker2}
    (\lambda-1) v^2F+(2\lambda+2)uvG+((\lambda-1)u^2+v^2)H=0,
\end{equation}
from which it follows that $u$ divides $(\lambda-1) F+H$ and $v$ divides $H$. Then we can write $H=vH'$ and $uF'=(\lambda-1)F+H$, for forms $F',H'$ of degree 5. It follows from this that $G=\frac{(\lambda-1)uH'-vF'}{2\lambda+2}$. Thus, we have the kernel relations: 
$$
\begin{pmatrix}
\frac{u}{\lambda-1} & \frac{v}{1-\lambda} \\ 
-\frac{v}{2\lambda+2} & \frac{(1-\lambda)u}{2\lambda+2}\\
0 & v
\end{pmatrix}
\begin{pmatrix}F' \\ H'\end{pmatrix}
=\begin{pmatrix}F \\ G \\ H\end{pmatrix}
$$
for some degree 5 forms $F',H',$ so that $\mathcal{N}_{C_4/Q}\cong\O_{\P^1}(5)^{\oplus2}.$\\ 

If $\lambda=-1$, then by (\ref{ker2}) we have that $v^2$ divides $H$, which gives the following kernel relations:
$$
\begin{pmatrix}0 & \frac{v^2}{2}-u^2 \\
1 & 0 \\
0 & v^2\end{pmatrix}
\begin{pmatrix}G \\ H'\end{pmatrix}
=\begin{pmatrix}F \\ G \\ H\end{pmatrix}
$$
for some degree 4 form $H'$, and since the above matrix is rank 2, we have that $\mathcal{N}_{C_4/Q}\cong\O_{\P^1}(4)\oplus\O_{\P^1}(6).$\\  

\subsection*{Case (3)} Then the morphism $\psi_f$ is given by the matrix: 
$$
\begin{pmatrix}-v^2 & 2uv+v^2 & uv-u^2\end{pmatrix}.$$

If $(F,G,H)\in\mathrm{ker}\,\psi_f$, then we have the relation 
\begin{equation}
    \label{ker3}
    -v^2F+(2uv+v^2)G+(uv-u^2)H=0,
\end{equation}
from which we see that $v$ divides $H$, so we may write $H=vH'$ for some form $H'$ of degree $4$. From substituting this expression back into (\ref{ker3}) and cancelling $v$, we see that $v$ divides $2G-uH'$, so that $G=\frac{vG'+uH'}{2}$ for some form $G'$ of degree 4. Thus we obtain the following kernel relations:
$$
\begin{pmatrix}u+\frac{v}{2} & \frac{3u}{2} \\
\frac{v}{2} & \frac{u}{2} \\
0 & v\end{pmatrix}
\begin{pmatrix}F' \\ G'\end{pmatrix}
=\begin{pmatrix}F \\ G \\ H\end{pmatrix}.
$$
So we see that $\mathcal{N}_{C_4/Q}\cong\O_{\P^1}(5)^{\oplus2}$.
\end{proof}

As a consequence of Theorem \ref{normal bundles}, we have the following classification of the exceptional divisor $E$ of $\pi\colon X\to Q$, which refines the statement of Lemma \ref{lemma:normal-bundle}:

\begin{corollary}
Let $\pi\colon X\to Q$ be a smooth Fano threefold in the family \textnumero2.21, and $E$ the $\pi$-exceptional divisor. Then for a general $X$, we have $E\cong\P^1\times\P^1$. Otherwise, $E\cong\F_2$.
\end{corollary}
\printbibliography
\end{document}